\documentclass[11pt]{amsart}
\usepackage{latexsym}
\usepackage{amssymb,amsmath,mathabx, amscd}
\usepackage[pdftex]{graphicx}
\usepackage{enumerate}

\usepackage{colonequals}
\usepackage{comment}
\usepackage{tikz}
\usepackage{tikz-cd}
\usepackage[margin=1in]{geometry}
\usepackage[title]{appendix}
\usepackage{fancyhdr}
\usepackage{hyperref, mathrsfs}

\newcommand{\cc}{\mathbb{C}}
\renewcommand{\O}{\mathcal{O}}
\newcommand{\rr}{\mathbb{R}}
\newcommand{\pp}{\mathbb{P}}

\newcommand{\qq}{\mathbb{Q}}
\newcommand{\zz}{\mathbb{Z}}

\renewcommand{\ss}{\mathbb{S}}

\renewcommand{\O}{\mathcal{O}}

\renewcommand{\L}{\mathcal{L}}

\renewcommand{\L}{\mathcal{L}}

\newcommand{\Hom}{\operatorname{Hom}}

\newcommand{\ch}{\operatorname{ch}}
\newcommand{\rk}{\operatorname{rk}}

\newcommand{\Jac}{\operatorname{Jac}}
\newcommand{\Pic}{\operatorname{Pic}}

\newcommand{\gon}{\operatorname{gon}}
\newcommand{\Exc}{\operatorname{Exc}}
\newcommand{\Amp}{\operatorname{Amp}}
\newcommand{\NS}{\operatorname{NS}}

\newcommand{\Sym}{\operatorname{Sym}}
\newcommand{\airr}{\operatorname{a.\!irr}}

\newcommand{\res}{\text{res}}

\newcommand{\btimes}{\boxtimes}

\renewcommand{\bar}{\overline}

\newcommand{\ra}{\rightarrow}

\makeatletter
\def\legendre@dash#1#2{\hb@xt@#1{%
  \kern-#2\p@
  \cleaders\hbox{\kern.5\p@
    \vrule\@height.2\p@\@depth.2\p@\@width\p@
    \kern.5\p@}\hfil
  \kern-#2\p@
  }}
\def\@legendre#1#2#3#4#5{\mathopen{}\left(
  \sbox\z@{$\genfrac{}{}{0pt}{#1}{#3#4}{#3#5}$}%
  \dimen@=\wd\z@
  \kern-\p@\vcenter{\box0}\kern-\dimen@\vcenter{\legendre@dash\dimen@{#2}}\kern-\p@
  \right)\mathclose{}}
\newcommand\legendre[2]{\mathchoice
  {\@legendre{0}{1}{}{#1}{#2}}
  {\@legendre{1}{.5}{\vphantom{1}}{#1}{#2}}
  {\@legendre{2}{0}{\vphantom{1}}{#1}{#2}}
  {\@legendre{3}{0}{\vphantom{1}}{#1}{#2}}
}
\def\dlegendre{\@legendre{0}{1}{}}
\def\tlegendre{\@legendre{1}{0.5}{\vphantom{1}}}
\makeatother

\newtheorem{thm}{Theorem}[section]
\newtheorem{ithm}{Theorem}
\newtheorem{lem}[thm]{Lemma}

\newtheorem{prop}[thm]{Proposition}

\newtheorem{cor}[thm]{Corollary}
\newtheorem{icor}[ithm]{Corollary}

\theoremstyle{definition}

\newtheorem{rem}[thm]{Remark}
\newtheorem{irem}[ithm]{Remark}

\newcommand{\defi}[1]{\textsf{#1}}

\newcommand\blfootnote[1]{%
  \begingroup
  \renewcommand\thefootnote{}\footnote{#1}%
  \addtocounter{footnote}{-1}%
  \endgroup
}

\title{Low Degree Points on Curves}
\author{Geoffrey Smith}
\author{Isabel Vogt}
\date{\today}

\begin{document}
\begin{abstract}
In this paper we investigate an arithmetic analogue of the gonality of a smooth projective curve $C$ over a number field $k$: the minimal $e$ such there are infinitely many points $P \in C(\bar{k})$ with $[k(P):k] \leq e$.  Developing techniques that make use of an auxiliary smooth surface containing the curve, we show that this invariant can take any value subject to constraints imposed by the gonality.  Building on work of Debarre--Klassen, we show that this invariant is equal to the gonality for all sufficiently ample curves on a surface $S$ with trivial irregularity.
\end{abstract}

\maketitle
\blfootnote{This is a pre-copyedited, author-produced version of an article accepted for publication in IMRN following peer review. The version of record, "Low degree points on curves", Int. Math. Res. Not.  \textbf{2022.1} (2022), 422--445, is available online at: https://doi.org/10.1093/imrn/rnaa137}

\section{Introduction}

Let $C$ be a nice (smooth, projective and geometrically integral) curve over a number field $k$.  For an algebraic point $P \in C(\bar{k})$, the \defi{degree of $P$} is the degree of the residue field extension $[k(P) : k]$. In this paper we investigate the sets
\[C_e \colonequals \left\{ P \in C(\bar{k}) : \deg(P) \leq e \right\} =  \bigcup_{[F:k] \leq e} C(F)\]
of algebraic points on $C$ with residue degree bounded by $e$.  

When $e=1$, this is the set of $k$-rational points on $C$.  If the genus of $C$ is 0 or 1, then there is always a finite extension $K/k$ of the base field over which $(C_K)_1 = C(K)$ is infinite.  On the other hand, if the genus of $C$ is at least 2, then for all finite extensions $K/k$, Faltings' theorem guarantees that the set $(C_K)_1$ is finite \cite{faltings2}.  While understanding the set of rational points is an interesting and subtle problem, here we will be primarily concerned with the infinitude of the sets $C_e$ as $e$ varies.  Define the \defi{arithmetic degree of irrationality} to be
\[\airr_k(C) \colonequals \min(e : \text{$C_e$ is infinite}).\]
This invariant is \emph{not} preserved under extension of the ground field, 
%(e.g.,~a genus $0$ curve has infinitely many $k$-points if and only if it has at least one), 
so we also define 
\[\airr_{\bar{k}}(C) \colonequals \min(e  : \text{there exists a finite extension $K/k$ with $(C_K)_e$ infinite}).\]
As is implicit in the notation, this notion depends only upon the $\bar{k}$-isomorphism class of $C$, see Remark \ref{geometric_property}.
The situation for $k$-points can therefore be summarized as: 
\[\airr_{\bar{k}}(C) = 1 \Leftrightarrow \text{genus of $C \leq 1$}\]

For $e \geq 2$, the situation for higher genus curves is more interesting.  Recall that the \defi{$k$-gonality} of $C/k$,
\[ \gon_k(C) \colonequals \min(e  : \text{there exists a dominant map $C \to \pp^1_{k}$ of degree $e$} ), \] 
is a measure of the ``geometric degree of irrationality" of $C$.  This notion is also not invariant under extension of the base field (e.g.,~a genus $0$ curve has $k$-gonality $1$ if and only if it has a $k$-point).  For that reason we also define the \defi{geometric gonality} to be $\gon_{\bar{k}}(C) \colonequals \gon_{\bar{k}}(C_{\bar{k}})$, which is stable under algebraic extensions.  If $f \colon C \to \pp^1_k$ is dominant of degree at most $e$, then $f^{-1}(\pp^1(k)) \subset C_e$.  Therefore we always have the upper bound
\begin{equation}\label{upper_bound}
\airr_k(C) \leq \gon_k(C). \end{equation}
This bound need not always be sharp: if $f \colon C\to E$ is a dominant map of degree at most $e$ onto a positive rank elliptic curve $E$, then $f^{-1}(E(k)) \subset C_e$ is also infinite.  When $e=2$ (resp.~$e=3$) then Harris--Silverman and Hindry \cite{harris_silverman, hindry} (resp.~Abramovich--Harris \cite{harris_abramovich}) showed 
\[\airr_{\bar{k}}(C) = e \Leftrightarrow \text{$e$ is minimal such that $C_{\bar{k}}$ is a degree $e$ cover of a curve of genus $\leq 1$}. \]
Debarre--Fahlaoui \cite{debarre_2} gave examples of curves lying on projective bundles over an elliptic curve that show the analogous result is \textit{false} for all $e \geq 4$.  The arithmetic degree of irrationality is therefore a subtle invariant of a curve, capturing more information than only low degree maps.

Implicit in the work of Abramovich--Harris \cite{harris_abramovich} and explicit in a theorem of Frey \cite{frey}, is the fact that Faltings' theorem implies that if $C_e$ is infinite, then $C$ admits a map of degree at most $2e$ onto $\pp^1_k$.  Therefore we have an inequality in both directions
\begin{equation}\label{range} \gon_k(C)/2 \leq \airr_k(C) \leq \gon_k(C). \end{equation}

In this paper, we develop and apply geometric techniques to compute $\airr_k(C)$ and $\gon_k(C)$ when $C$ lies on a smooth auxiliary surface $S$.  The first result in this direction is that the inequalities in \eqref{range} are sharp, and that subject to these bounds, we may decouple $\airr_k(C)$ and $\gon_k(C)$.

\begin{ithm}\label{every_value}
Given any number field $k$ and a pair of integers $\alpha, \gamma \geq 1$, there exists a nice curve $C/k$ such that 
\begin{equation}\label{alpha_gamma_equalities}\airr_k(C) = \airr_{\bar{k}}(C) = \alpha, \qquad \gon_k(C) = \gon_{\bar{k}}(C) = \gamma\end{equation}
if and only if $\gamma/2 \leq \alpha \leq \gamma$.
In fact, for $\gamma \geq 4$, the equalities \eqref{alpha_gamma_equalities} are satisfied for all smooth curves in numerical class $(\gamma, \alpha)$ on $S = E\times \pp^1_k$, where $E/k$ is a positive-rank elliptic curve.
\end{ithm}

Using these geometric techniques, we next describe classes of curves where the arithmetic and geometric degrees of irrationality agree; that is, where there are as few points as allowed by the gonality.  In such cases, we have the strongest finiteness statements on low degree points.
% and inequality \eqref{upper_bound} is an equality.

The first explicit examples of this kind were given by Debarre and Klassen for smooth plane curves $C/k$ of degree $d$ sufficiently large.  Max Noether calculated the gonality for $d \geq 2$:
\begin{enumerate}
\item If $C(k) \neq \emptyset$, then $\gon_k(C) = d-1$, and all minimal degree maps are projection from a $k$-point of $C$, and
\item If $C(k) = \emptyset$, then $\gon_k(C) = d$.
\end{enumerate}
For smooth plane curves of degree $d \geq 8$, Debarre--Klassen \cite{debarre} prove an arithmetic strengthening of this result:
\begin{enumerate}
\item If $C(k) \neq \emptyset$, then $C_{d-2}$ is finite, and so $\airr_k(C) = \gon_k(C) = d-1$.  Furthermore, all but finitely many points of degree $d-1$ come from intersecting $C$ with a line over $k$ through a $k$-point of $C$.
\item If $C(k) = \emptyset$, then $C_{d-1}$ is finite, and so $\airr_k(C) = \gon_k(C) = d$.
\end{enumerate}

We generalize this result to smooth curves on other surfaces $S$.  The key property of $\pp^2$ that we need in general is that it has discrete Picard group; i.e., in the classical language of surfaces, it has irregularity $0$.  The explicit condition $d\geq 8$ can be replaced by requiring that the class of $C$ is ``sufficiently positive" in the ample cone in the sense that it is sufficiently far from the origin, and sufficiently far from the boundary of the ample cone.

\begin{ithm}\label{main_q0}
Let $S/k$ be a nice surface with $h^1(S, \O_S)=0$. If $C/k$ is a smooth curve in an ample class on $S$, then  
%$C_{\gon_k(C)-1}$ is finite, so
\[\airr_k(C) \geq \min\left( \gon_k(C), \frac{C^2}{9}\right). \]
In particular, let $P$ be a very ample divisor on $S$, and define the set
\[\Exc_P \colonequals \big\{\text{integral classes $H$ in $\Amp(S)$ such that } H^2 \leq 9(H \cdot P-1)\big\}.\]
\begin{enumerate}
\item If $C \subset S$ is a smooth curve with class $[C] \in \Amp(S) \smallsetminus \Exc_P$, then $\airr_k(C) = \gon_k(C)$.
\item For any closed subcone $N \subseteq \Amp(S)$, the set
$\Exc_P(N) \colonequals \Exc_P \cap N$
is finite.
\end{enumerate}
\end{ithm}

As an immediate consequence, we obtain an effective generalization of the Debarre-Klassen result to other surfaces with $h^1(S,\O_S)=0$.

\begin{icor}\label{main_pic1}
Suppose that $C$ embeds in a nice surface $S/k$ having $h^1(S,\O_S)=0$, with $\O_S(1)$ very ample and $C \in |\O_S(\alpha)|$.  If 
\[\alpha \geq \begin{cases} 8 & : \ \O_S(1)^2=1\\ 9 & : \ \text{otherwise,} \end{cases}\]  
then $\airr_k(C)  = \gon_k(C)$.
\end{icor}

\begin{icor}
Under the hypotheses of Corollary \ref{main_pic1}, if $S$ satisfies $\mathrm{Pic}(S_{\overline{k}})=\zz \cdot \O_S(1)$, then there are finitely many points of degree strictly less than $(\alpha-1)\O_S(1)^2$ on $C_K$ for any finite extension $K/k$.
\end{icor}
\begin{proof}
By \cite[Lemma 4.4]{ullery}, $(\alpha-1)\O_S(1)^2 \leq \gon_{\bar{k}}(C) \leq \gon_k(C)$.  Therefore by Corollary \ref{main_pic1},
\[(\alpha-1)\O_S(1)^2 \leq \airr_k(C). \qedhere\]
\end{proof}

Corollary \ref{main_pic1} combined with \cite[Theorem 3.1]{ullery} is enough to deduce the analogous result for most complete intersection curves in $\pp_k^n$, generalizing in another direction Debarre and Klassen's original result when $n=2$:

\begin{icor}\label{main_ci}
Let $C/k$ be a smooth complete intersection curve in $\pp_k^n$, $n \geq 3$, of type $9 \leq d_1 < d_2\leq \cdots \leq d_{n-1}$.  Then 
\[\airr_k(C) = \gon_k(C). \]
In particular, by Lazarsfeld's computation of the minimal gonality of such a curve \cite[Exercise 4.12]{lazars}, there are finitely many points of degree strictly less than $(d_1-1)d_2 \cdots d_{n-1}$ on $C_K$ for any finite extension $K/k$.
\end{icor}

%\begin{irem}\label{effectively_computable}
%The proof of Theorem \ref{main_q0} is constructive; in Appendix \ref{sec:effectively_computable} we show that if $N$ is rational polyhedral, the set $\Exc(N)$ is effectively computable in terms of the divisor structure on the surface $S$; we direct the interested reader there for precise statements.
%%, a very ample polarization $P \in N$, generators of $N$ and $N^1(S)$, and the intersection product on $N^1(S)$.
%\end{irem}

For any surface $S$ and any finite polyhedral subcone $N \subseteq \Amp(S)$, the set $\Exc_P(N)$ in Theorem \ref{main_q0} is effectively computable.
Given some particular surface $S$, our techniques are amenable to explicit computations, and can sometimes yield a full computation of all classes $[C] \in \Amp(S)$ for which $\airr_k(C)$ is strictly less than $\gon_k(C)$.  For example:

\begin{ithm}\label{main_p1xp1}
Let $C$ be a nice curve of type $(d_1, d_2)$, with $1 \leq d_1 \leq d_2$, on $\pp^1_k \times \pp^1_k$.  Then if $(d_1,d_2) \neq (2,2)$ or $(3,3)$, we have that $\airr_k(C) = \gon_k(C) = d_1$.  In particular, this lets us compute:
\[ \airr_{\bar{k}}(C) = \begin{cases} 1 &: \ d_1 \leq  1 \text{ or } (d_1,d_2) = (2,2), \\
2 & : \ d_1 = 2 \text{ and } d_2 \geq 3, \text{ or } (d_1,d_2) = (3,3) \text{ and $C$ bielliptic,} \\
d_1 & : \ \text{otherwise}. \end{cases} \]
\end{ithm}

\begin{irem}
We say that a point $P \in C(\bar{k})$ is \defi{sporadic} if $[k(P):k] < \airr_k(C)$.
From the perspective of the arithmetic of elliptic curves, there is much interest in understanding sporadic points on modular curves, e.g. the classical $X_1(N)$, since these indicate ``usual" level structure.  Since $X_1(N)(\qq) \neq \emptyset$, $X_1(N)$ is always a subvariety of its Jacobian variety $J_1(N)$.  And whenever $\#J_1(N)(\qq)$ is finite, we have that $\airr_\qq(X_1(N)) = \gon_\qq(X_1(N))$.  In particular this holds for $N\leq 55$ and $N \neq 37, 43, 53$ by work of Derickx and van Hoeij \cite{DvH}; in the same article, they compute the gonality (and therefore the arithmetic degree of irrationality when $N \neq 37$) for all $N\leq 40$.  It is our hope that the geometric techniques we develop here might prove useful for specific curves of arithmetic interest.
\end{irem}

As in previous work, the proofs of these results begin by translating the problem of understanding degree $e$ points on $C$ to understanding rational points on $\Sym^eC \equalscolon C^{(e)}$, which is a parameter space for effective divisors of degree $e$ on $C$.  There is a natural map
\[C^{(e)} \to \Pic^eC, \]
sending an effective divisor $D$ to the class of the line bundle $\O(D)$.  We denote the image of this map $W_e(C)$.  We now have two problems: understand the infinitude of rational points on the fibers of $C^{(e)} \to \Pic^e C$ (which is related to the dimension of the space of sections of the corresponding line bundle), and understand the infinitude of rational points on the image $W_e(C)$ (which, by Faltings' Theorem, is related to positive-dimensional abelian varieties in $W_e(C)$).

The majority of this paper is therefore devoted to proving purely geometric results over $\cc$ about nonexistence of positive-dimensional abelian varieties in $W_e(C)$ for appropriate $e$.  Using the theory of stability conditions on vector bundles, we show that such an abelian variety in $W_eC$ forces the existence of a certain type of effective divisor on $S$.  Given a particular surface $S$, we can often use the geometry of $S$ to obtain a contradiction; this is how we proceed with Theorem \ref{every_value}.  When the surface is not explicitly given, the fact that such a divisor class does not move in a positive dimensional family (from $h^1(S, \O_S)=0$) allows us to construct an embedding of the abelian variety into $W_fC$ for smaller $f$ and eventually obtain a contradiction.

\section{Abelian Varieties in \texorpdfstring{$W_eC$}{We(C)}}

In this section we prove purely geometric results (Theorems \ref{every_value_geometric} and \ref{main_geometric}) about nonexistence of abelian subvarieties that will imply our main theorems.  Therefore the basefield is assumed to be $\cc$, unless otherwise noted, and $\gon(C) \colonequals \gon_{\bar{k}}(C)$ denotes the geometric gonality.

The proofs of these results will proceed by contradiction: the existence of a positive-dimensional abelian variety $A \subset W_eC$ will force the existence of a family of effective divisors of moderately low degree moving in basepoint-free pencils.  We will then use a geometric lemma proved in  Section \ref{low_deg} to produce interesting effective divisors on an auxiliary surface containing the curve.  The proof ideas bifurcate here: when the auxiliary surface is specified explicitly, we may then directly use the geometry to obtain a contradiction.  When the surface is simply known to have $h^1(\O)=0$, we use the interesting effective divisor to inductively produce such a family of effective line bundles on $C$ of even lower degree that will force a contradiction for all but finitely many possible starting classes of curves $C$.

The first step in this procedure relies on the following observation, due originally to Abramovich and Harris \cite[Lemma 1]{harris_abramovich}, and whose consequence for the gonality of $C$ was noted by Frey \cite{frey}.  Assume that $A \subset W_e C$ is a translate of an abelian variety of dimension at least $1$ and $A \not\subset x + W_{e-1}C$ for any $x \in C$.  Let $A_2$ denote the image of $A \times A$ under the addition map $W_eC \times W_eC \to W_{2e}C$.  Note that $A_2$ is (noncanonically) isomorphic to $A$: a choice of basepoint in $A$ induces an isomorphism $\Pic^eC \simeq \Jac_C$, under which the addition map on $W_eC$ agrees with the group law on $\Jac_C$ and $A \subset \Jac_C$ is an abelian subvariety. 
\begin{lem}\label{produce_bpf_pencil}
The line bundle $L_p$ corresponding to a point $p\in A_2 \subseteq W_{2e}C \subseteq \Pic^{2e}C$ is basepoint-free and has 
\[h^0(C, L_p) -1 \geq \dim A.\]
\end{lem}
\begin{proof}
Let $C^{(e)}$ denote the $e$th symmetric power of $C$.  We have the following commutative diagram
\begin{center}
\begin{tikzcd}
C^{(e)} \times C^{(e)} \arrow{r}{\text{finite}}[swap]{\text{dom.}} \arrow{d} & C^{(2e)} \arrow{d} \\
W_eC \times W_eC \arrow{r} & W_{2e}C	\\
A \times A \arrow[draw=none]{u}[sloped,auto=false]{\subseteq} \arrow{r} & A_2 \arrow[draw=none]{u}[sloped,auto=false]{\subseteq}
\end{tikzcd}
\end{center}
Given a point $p \in W_{2e}C$, the fibers of the map $C^{(2e)} \to W_{2e}C$ are of dimension $h^0(C,L_p) - 1$.  As the fibers of the bottom map $A \times A \to A_2$ are $(\dim A)$-dimensional, we see that if $p \in A_2$,  the fiber of $C^{(2e)} \to W_{2e}C$ over $p$ must be at least this large.  Furthermore, if $x \in C$ is in the base locus of this $(\dim A)$-dimensional linear system, then it would necessarily be the case that $x$ is always in the linear system parameterized by the points of $A$.  This is impossible, as we assumed that $A$ is not contained in a translate of $W_{e-1}C$.
\end{proof}

\subsection{Linear series of low degree}\label{low_deg}
In this section we prove the key geometric input on linear series of moderately low degrees on curves $C$ whose class is ample on a surface $S$.  This is a purely geometric result over an algebraically closed field of characteristic $0$.

We first recall some of the basic theory of torsion-free coherent sheaves on  varieties over algebraically closed fields $k = \bar{k}$.
 Let $F$ be a torsion-free coherent sheaf on a nice variety $X$ of dimension $m$. Given an ample class $H$ on $X$, we define the \defi{slope of $F$ with respect to $H$} 
 \[ \mu_H(F):=\frac{c_1(F)\cdot H^{m-1}}{\mathrm{rk}(F)}.\]
In what follows, we leave the reliance on $H$ implicit, and just refer to the slope as $\mu(F)$. 

The sheaf $F$ is called \defi{$\mu$-unstable (with respect to $H$)} if there exists a coherent sheaf $E\subseteq F$ such that 
\[ \mu(E) >\mu(F). \]
Otherwise we say that $F$ is \defi{$\mu$-semistable (with respect to $H$)}.

The $\mu$-semistable sheaves are the building blocks of torsion-free coherent sheaves on $X$. More precisely, for $F$ any torsion-free coherent sheaf, by \cite[Theorem 1.6.7]{HL10} there exists a unique \defi{Harder-Narasimhan filtration} of $F$,
\[ 0=F_0\subset F_1\subset \ldots\subset F_n=F,\]
which is characterized by the following properties
\begin{enumerate}
\item Each quotient $G_i\colonequals F_i/F_{i-1}$ is a torsion free $\mu$-semistable sheaf.
\item  If $1\leq i < j \leq n$, then $\mu(G_i)>\mu(G_j)$.
\end{enumerate}
In particular, we will use the fact that given an unstable
%--that is, a non-semistable--
torsion free coherent sheaf $F$, there is a unique nonzero subsheaf $E\subset F$ such that $F/E$ is semistable and torsion free, and $\mu(E)$ is maximal among subsheaves of $F$. We call this $E$ the \defi{maximal destabilizing subsheaf} of $F$.

\begin{rem}\label{assumptions_destab}
If $X$ is a curve (i.e., $m=1$), then a vector bundle $F$ is unstable if and only if it is destabilized by a subbundle $E \subseteq F$, since the saturation of a destabilizing subsheaf will yield a destabilizing subbundle.  However, saturation does not in general yield a subbundle (though, by the fact that the quotient is torsion free, the maximal destabilizing subsheaf of a vector bundle is always saturated \cite[Definition 1.1.5]{HL10}).  If $X$ is a surface (i.e., $m=2$) and $F$ is a vector bundle, then the maximal destabilizing subsheaf is itself also a vector bundle (though not in general a subbundle), as we now show.  By \cite[Corollary 1.4]{harts80}, any sheaf $E$ on $X$ that is \defi{reflexive} (i.e. the natural map $E \to E^{\vee \vee}$ is an isomorphism) is locally free.  Furthermore, a saturated subsheaf $E \subseteq F$ of a locally free sheaf is reflexive (as $E^{\vee \vee}/E \hookrightarrow F/E$ would otherwise be a torsion subsheaf, see also \cite[Corollary 1.5]{harts80}).  The maximal destabilizing subsheaf is saturated, and hence reflexive, and hence locally free.
\end{rem}

 We also define the \defi{discriminant} of a coherent sheaf $F$ on a smooth complex projective surface in terms of Chern characters as the quantity
 \[ \Delta(F):=2\ch_0(F)\ch_2(F)-\ch_1(F)^2.\]
 %changed the sign
 
 The following fundamental theorem of Bogomolov \cite[Theorem 3.4.1]{HL10} implies that the property of $\mu$-stability of sheaves on surfaces is numerical.
 
\begin{thm}[Bogomolov inequality]\label{bogomolov}
Let $S$ be a smooth complex projective surface.
% and let $H$ be an ample divisor on $S$. 
If $F$ is a $\mu$-semistable torsion-free coherent sheaf on $S$ with respect to some ample class, then $\Delta(F) \leq 0$.
\end{thm}

\begin{rem}
Once one knows that stability is a numerical property, the fact that $\Delta(F)$ is the precise combination of Chern classes capturing this follows from the fact that it is the minimal polynomial in the Chern classes that is invariant under twisting by line bundles.
\end{rem}

We now apply Bogomolov's Inequality to prove a geometric result that will ultimately produce the bounds we desire.  The following result is originally due to Reider \cite{reider2}, see Prop. 2.10,
Remark 2.11 and Cor. 1.40, but we include a proof for completeness.

\begin{prop}\label{lin_series_prop}
Let $S$ be a smooth projective surface and $C \subset S$ a smooth curve such that $\O_S(C)$ is ample.  If  $\Gamma$ is a divisor on $C$ that moves in a basepoint-free pencil, satisfying
\[ \deg \Gamma < C^2/4,\]
then there exists a divisor $D$ on $S$ satisfying the following four conditions 
\begin{enumerate}
\item $h^0(S, D) \geq 2$,
\item $C \cdot D < C^2/2$,
\item $\deg \Gamma \geq D \cdot (C - D)$.
\item If $E$ is any divisor on $S$ such that
\begin{equation}\label{eConditions}h^0(\O_C(E|_C - \Gamma)) = 0 \qquad \text{and} \qquad E \cdot C < C^2, \end{equation}
then $h^0(\O_S(E - D)) = 0$.  In particular, $h^0(\O_C(D|_C - \Gamma)) > 0$.
\end{enumerate}

\end{prop}

\begin{proof}
As $\Gamma$ moves in a basepoint-free pencil, there is a choice of two sections generating the line bundle and hence giving a surjection 
$\O_C^{\oplus 2} \to \O_C(\Gamma)$.  This map fits into an exact sequence
\begin{equation}\label{exactC} 0 \to \O_C(-\Gamma) \to O_C^{\oplus 2} \to \O_C(\Gamma) \to 0. \end{equation}

Let $i \colon C \hookrightarrow S$ be the inclusion map.  Then $i_*\O_C(\Gamma)$ is a torsion sheaf on $S$.  We define the coherent sheaf $F$ on $S$ via the exact sequence of coherent sheaves
\begin{equation}\label{exactS} 0 \to F \to \O_S^{\oplus 2} \to i_*\O_C(\Gamma) \to 0,\end{equation}
where the right map factors through the surjection 
$\O_S^{\oplus 2 } \to \O_C^{\oplus 2}$.

As the only associated point of $i_*\O_C(\Gamma)$ is the generic point of the divisor $C \subseteq S$, the sheaf $F$ is reflexive \cite[Corollary 1.5]{harts80}, and hence locally free \cite[Corollary 1.4]{harts80}.

Set $e \colonequals \deg \Gamma$.  Using Grothendieck-Riemann-Roch to calculate the Chern classes of the pushforward $i_*\O_C(\Gamma)$, we may compute the discrete invariants of $F$ from the exact sequence \eqref{exactS}:
\begin{align*}
\ch_0(F) &= \rk(F) = \rk(\O_S^{\oplus 2}) = 2 \\
\ch_1(F) &= c_1(F) = c_1(\O_S^{\oplus 2}) - c_1(i_*\O_C(\Gamma)) = -[C] \\
\ch_2(F) &= \ch_2(\O_S^{\oplus 2}) - \ch_2(i_*\O_C(\Gamma)) = C^2/2 - c_2(i_*\O_C(\Gamma)) = C^2/2 - e.
\end{align*}
The vector bundle $F$ therefore has Chern character $\ch(F)=(2, -[C], C^2/2-e)$ and hence has discriminant $\Delta(F)=C^2-4e$. Therefore, by assumption, $\Delta(F)>0$, so $F$ is $\mu$-unstable with respect to any ample class on $S$; we will use $C$ as the ample class on $S$. 

Let $L$ to be the maximal destabilizing subsheaf of $F$, which by Remark \ref{assumptions_destab} is locally free and hence a line bundle.
 Write $L\cong \O(-D)$, where $D$ is some divisor on $S$. We show this $D$ satisfies properties (1)--(4) of the proposition. 
\begin{itemize}
\item \textbf{Property (2): $ C \cdot D < C^2 /2$:}

By the definition of the maximal desabilizing subsheaf, we have
 \[ \mu_C(L) = (-D) \cdot C > (-C) \cdot C/2 = \mu_C(F),\]
 which is equivalent to property (2).
 \item \textbf{ Property (3): $e\geq D \cdot (C-D)$:}
 
In the exact sequence
\begin{equation}\label{2} 0\ra\O(-D)\ra F\ra Q\ra 0,\end{equation}
the quotient $Q$ is $\mu$-semistable with respect to $C$. 

Therefore $\Delta(Q)\leq 0$, which is equivalent to
\[e\geq D \cdot (C-D).\]
\item \textbf{Property (1): $h^0(S,D)\geq 2$:} 

Dualizing the inclusion $\O(-D) \to \O_S^{\oplus 2}$, it suffices to show that the map $H^0(S, \O_S^{\oplus 2}) \to H^0(S, \O(D))$ is injective.  If it is not injective, then we may assume that one map $H^0(S, \O_S) \to H^0(S, \O(D))$ is zero, and hence (since $\O_S$ and $\O(-D)$ are reflexive), the original inclusion must factor $\O(-D) \to \O_S \to \O_S^{\oplus 2}$.  In particular, $D$ is effective and the quotient of the inclusion $\O(-D) \to \O_S^{\oplus 2}$ is isomorphic to $\O_S \oplus \O_D$.

Because $C\cdot D<C^2/2$ and $C$ is integral and ample, we have that $D\cap C$ must be zero-dimensional.  Hence $\Hom(\O_D, i_*\O_C(\Gamma)) = 0$.  Furthermore, $\O_S$ admits no surjective maps onto $i_*\O_C(\Gamma)$.  Therefore $\O_D \oplus \O_S$ does not surject onto $i_*\O_C(\Gamma)$.  This is a contradiction, as the inclusion $\O(-D) \hookrightarrow \O_S^{\oplus 2}$ factors through $F \hookrightarrow \O_S^{\oplus 2}$, and $\O_S^{\oplus 2}/F = i_*\O_C(\Gamma)$.
Therefore, we must have $h^0(S,\O(D))\geq 2$. 

\item \textbf{Property (4): A divisor $E$ satisfying equation (\ref{eConditions}) also satisfies $h^0(\O_S(E-D))=0$:}

Let $E$ be a divisor on $S$ such that
\[h^0(\O_C(E|_C - \Gamma)) = 0 \qquad \text{and} \qquad E \cdot C < C^2. \]
 By the projection formula we have 
\[\O_S(E) \otimes i_* \O_C( \pm \Gamma) \simeq i_* \O_C(E|_C \pm \Gamma).\]  
We therefore have the diagram with exact rows
\begin{center}
\begin{tikzcd}
0  \arrow{r} & F \otimes \O_S(E)  \arrow{r} & \O_S(E)^{\oplus 2} \arrow{d}{\res} \arrow{r} &  i_*\O_C(\Gamma) \otimes \O_S(E) \arrow[equal]{d} \arrow{r} & 0 \\
0 \arrow{r} & i_*\O_C(E|_C - \Gamma) \arrow{r} & i_*\O_C(E|_C)^{\oplus 2} \arrow{r} & i_*\O_C(E|_C + \Gamma) \arrow{r} & 0
\end{tikzcd}
\end{center}
By assumption we have $E \cdot C < C^2$; therefore as $C$ is ample, $h^0(\O_S(E - C)) = 0$, and so the vertical map $\res$ is injective on global sections.  Combined with the assumption that $h^0(C, E|_C - \Gamma) =0$, we have that $h^0(S, F \otimes \O_S(E)) = 0$.  Tensoring \eqref{2} with $\O_S(E)$ and taking global sections, this implies $h^0(S, E - D) = 0$ as desired. 

Since $C \cdot D < C^2 /2 < C^2$, if $h^0(C, D|_C - \Gamma) =0$, then we could take $E = D$ and obtain the contradiction $h^0(S, \O_S) = 0$. Hence we must have that $\O_C(D|_C - \Gamma)$ \textit{is} effective. \qedhere
\end{itemize}
\end{proof}

\begin{rem}
The use of Bogomolov's inequality as a tool for proving the existence of divisors satisfying nice positivity properties originates in Reider's proof \cite{reider} of Reider's theorem, and has been developed by Lazarsfeld \cite{lazars} and others.  In particular, see \cite{ullery, ras} for recent applications in the Picard rank $1$ case.
\end{rem}

\subsection{Examples: curves on \texorpdfstring{$E \times \pp^1$}{E x P1}}
As a first example, let us now see how these techniques apply when $C$ is a smooth curve on $S = E \times \pp^1$.
We denote the projection maps
\begin{center}
\begin{tikzcd}
& E \times \pp^1 \arrow[swap]{dl}{\pi_1} \arrow{dr}{\pi_2} & \\
E & & \pp^1
\end{tikzcd}
\end{center}
to $E$ and $\pp^1$, respectively.  Then 
\[\Pic S = \pi_1^* \Pic E \oplus \pi_2^* \Pic \pp^1. \]
As is standard, if $\L_1$ is a line bundle on $E$ and $\L_2$ is a line bundle on $\pp^1$, we write $\L_1 \boxtimes \L_2$ for $\pi_1^* \L_1 \otimes \pi_2^*\L_2$.
%The class of the canonical divisor on $S$ is
%\[K_S = \pi_2^*K_{\pp^1} = \O_{\pp^1}(-2).\]
Furthermore, the N\'eron-Severi group is $\NS(S) = \zz \oplus \zz$, spanned by the classes $F_1$ and $F_2$ of fibers of the first and second projections, respectively.  These satisfy the intersection relations
\[F_1^2 = 0, \qquad F_2^2 = 0, \qquad F_1 \cdot F_2 = 1.\]
We will denote the numerical class $xF_1 + yF_2$ of a divisor by $(x,y)$.  The effective cone of $S$ is then the set of all classes with $x,y \geq 0$, and the ample cone is the set of all classes with $x,y > 0$.  

The following geometric result is the main ingredient in the proof of Theorem \ref{every_value}.

\begin{thm}\label{every_value_geometric}
Let $C$ be a smooth curve on $S = E \times \pp^1$ in numerical class $(\gamma, \alpha)$ for 
\[2 \leq \gamma/2 \leq \alpha \leq \gamma. \]
Then $C$ satisfies the following properties.
\begin{enumerate}
\item[(a)] $\gon(C) = \gamma$.
\item[(b)] $W_{\alpha} C$ contains an elliptic curve isogenous to $E$.
\item[(c)] If $e < \alpha$, then $W_eC$ does not contain any positive-dimensional abelian varieties.
\end{enumerate}
\end{thm}
\begin{rem}
Note that Bertini's theorem guarantees that there exist smooth curves in numerical class $(\gamma, \alpha)$ once $\gamma \geq 2$ and $\alpha \geq 1$, as the linear equivalence class is necessarily basepoint free.
\end{rem}

\begin{proof}  We have $\O_S(C) \simeq \O_E(\gamma e) \btimes \O_{\pp^1}(\alpha)$ for some point $e \in E$; then $C^2 = 2\alpha \gamma$.  The two projection maps exhibit $C$ as a $\gamma$-sheeted cover of $\pp^1$, and an $\alpha$-sheeted cover of $E$.  Therefore $\gon(C) \leq \gamma$.  Furthermore, we have a nonconstant map $E \to W_{\alpha}(C)$ sending $x \in E$ to $\O(\pi_1^{-1}(x))$, proving part (b).

\begin{enumerate}

\item[(a)] Suppose to the contrary that $\Gamma$ is a divisor on $C$ of degree at most $\gamma -1$ that moves in a basepoint free pencil.  Then
\[\deg \Gamma \leq \gamma -1 <  \alpha \gamma/2 = C^2/4,\]
as $\alpha \geq 2$.  So by Proposition \ref{lin_series_prop}, there exists an effective divisor $D$ on $S$ with at least $2$ sections, satisfying \\
\indent\ref{lin_series_prop}(2): $C \cdot D < C^2/2$; \\
\indent\ref{lin_series_prop}(3): $D \cdot (C -D) \leq \deg \Gamma$.\\
The divisor $D$ is in numerical class $xF_1 + y F_2$ for some $x \geq 0$ and $y \geq 0$, and so these numerical conditions translate into\\
$\alpha x + \gamma y < \alpha \gamma$ \\
$\alpha x +  \gamma y - 2xy \leq \Gamma < \gamma$. \\
Upon rearrangement we have:\\
(2'): $\alpha(\gamma/2 - x) + \gamma(\alpha/2-y) > 0$, \\ 
(3'): $(\gamma/2 - x)(\alpha/2 - y) > (\gamma/2)(\alpha/2 -1) \geq 0$,\\
as $\alpha\geq 2$.  Therefore both $\gamma/2 - x$ and $\alpha/2 - y$ have to be positive.  Furthermore, we have
\[(\gamma/2)(\alpha/2 - y)\geq (\gamma/2 - x)(\alpha/2 - y)  > (\gamma/2)(\alpha/2-1),\]
so $y = 0$.
Plugging $y=0$ back into inequality (3'), we see
\[(\alpha/2)(\gamma/2-x) > (\gamma/2)(\alpha/2-1)\]
 and so $x < \gamma/\alpha \leq 2$.
So $x$ is $0$ or $1$.  But every divisor of numerical class $0$ or $F_1$ has at most $1$ section, which is a contradiction.

\item[(c)] 
Suppose to the contrary that there exists a positive dimensional abelian variety $A \hookrightarrow W_eC$ for $e \leq \alpha -1$; and further that $e$ is minimal for this property.  Then by Lemma \ref{produce_bpf_pencil}, the points $p \in A_2$ parameterize basepoint free linear systems $\Gamma_p$.  On the other hand, 
\[\deg \Gamma_p = 2e \leq 2\alpha - 2 < \alpha\gamma/2 = C^2/4,\] 
since $\gamma \geq 4$.  Proposition \ref{lin_series_prop} produces a divisor $D_p$ on $E \times \pp^1$, say in numerical class $xF_1 + y F_2$, satisfying the following properties: \\
\indent\ref{lin_series_prop}(1): $h^0(D_p) \geq 2$; \\
\indent\ref{lin_series_prop}(2): $\alpha x + \gamma y < \alpha \gamma$; \\
\indent\ref{lin_series_prop}(3): $\alpha x  + \gamma y - 2xy \leq \deg \Gamma_p \leq 2\alpha -2 \leq 2\gamma -2$.\\
\indent\ref{lin_series_prop}(4): $\O_C(D_p|_C - \Gamma_p)$ is effective.\\
We may write the two inequalities as
\[\alpha(\gamma/2-x) + \gamma(\alpha/2-y) > 0, \qquad (\gamma/2-x)(\alpha/2-y) > \alpha(\gamma/4-1) \geq 0, \]
with the rightmost inequality coming from our assumption that $\gamma \geq 4$.  Therefore both $\gamma/2-x$ and $\alpha/2-y$ must be positive.  We have
\[(\gamma/2-x)(\alpha/2)\geq (\gamma/2-x)(\alpha/2-y)  > \alpha(\gamma/4 -1),\]  
and so $x < 2$.
Similarly for $y$ we obtain $y < 2\alpha/\gamma \leq 2$.  Combining this with the requirement that $D_p$ move in a pencil on $E \times \pp^1$, we see that it must be in numerical class $F_2$ or $F_1 + F_2$.

Let $D$ be a divisor on $S$ in numerical class $F_1 + F_2$ that contains $D_p$.  Let $q \in E$ be such that $\O_S(D) \simeq \O_E(q) \btimes \O_{\pp^1}(1)$.
Then
\[ D|_C - \Gamma_p \geq D_p|_C - \Gamma_p \geq 0,\]
by \ref{lin_series_prop}(4).
By the K\"unneth formula
\[H^0(S, D) = H^0(E \times \pp^1, \O_E(q)\btimes \O_{\pp^1}(1)) \simeq H^0(E , \O_E(q)) \otimes H^0(\pp^1, \O_{\pp^1}(1))\]
and so every divisor in $|D|$ is reducible, the union of the fiber $\pi_1^{-1}(q)$ and some fiber of $\pi_2$.  
As 
\[\O_S(D-C) \simeq \O_E(q-\gamma e)\btimes \O_{\pp^1}(1-\alpha),\] 
and both $H^0(E, \O_E(q-\gamma e))$ and $H^0(\pp^1, \O_{\pp^1}(1-\alpha))$ are $0$, the K\"unneth formula implies that $h^0(S, D-C) = h^1(S, D-C) = 0$. 
Therefore the map $H^0(E \times \pp^1, D) \to H^0(C, D|_C)$ is an isomorphism.
We therefore have $h^0(C, D|_C) = 2$ and every divisor on $C$ linearly equivalent to $D|_C$ is the union of $\pi_1^{-1}(q) \cap C$ and $\pi_2^{-1}(z) \cap C$ for some $z \in \pp^1$.  The linear system $\big|D|_C\big|$ has base locus exactly $\pi_1^{-1}(q) \cap C$.

By assumption $|\Gamma_p|$ is a basepoint free sub-linear series of $\big|D|_C\big|$.  As such, a general element of $|\Gamma_p|$ cannot pass through the basepoints of $\big|D|_C\big|$, and so must be supported in a fiber of the second projection $\pi_2 \colon C \to \pp^1$.  Therefore $\pi_2^*(\O_{\pp^1}(1))|_C -\Gamma_p$ is effective.  
Since $\Gamma_p$ is basepoint free, 
\[\deg(\Gamma_p) \geq \gon(C) = \gamma,\]
by part (a).
We also have $\deg(\pi_2^*( \O_{\pp^1}(1))|_C)=\gamma$, which forces $\Gamma_p = \pi_2^*(\O_{\pp^1}(1))|_C$ for all $p$.  Since $\Gamma_p$ is independent of $p$, the dimension of $A_2$ is $0$.  This contradicts the fact that $A$ has positive dimension. \qedhere

\end{enumerate}
\end{proof}
\subsection{Nonexistence of abelian subvarieties}

We now show how Lemma \ref{produce_bpf_pencil} in combination with Proposition \ref{lin_series_prop}, can prove the nonexistence of positive-dimensional abelian subvarieties in $W_eC$ when $C$ lies on an arbitrary smooth surface with $h^1(S, \O_S) = 0$ and $e$ is small. 

\begin{thm}\label{main_geometric}
Let $S/\cc$ be a nice surface with $h^1(S, \O_S)=0$, and let $C$ be a smooth ample curve on $S$. Then for $e < C^2/9$, the locus $W_eC$ contains no positive-dimensional abelian varieties.
\end{thm}

\begin{proof}
Suppose to the contrary that for some $e < C^2/9$, there exists a positive-dimensional abelian variety $A$ contained in $W_eC$.  Choose $e$ minimal.  By Lemma \ref{produce_bpf_pencil}, if $p$ is in $A_2 \subseteq W_{2e}C$, then the corresponding effective line bundle $\O(\Gamma_p)$ moves in a base point free pencil. 

By our hypothesis on $e$, we have that $2e<C^2/4$.  Applying Proposition \ref{lin_series_prop} to the divisor $\Gamma_p$ on $C$, there exists a divisor $D_p$ on $S$ satisfying:\\ 
\indent\ref{lin_series_prop}(2): $C \cdot D_p < C^2/2$; \\
\indent\ref{lin_series_prop}(3): $D_p \cdot (C -D_p) \leq \deg \Gamma_p =2e $; \\ 
\indent\ref{lin_series_prop}(4): $H^0(D_p\vert_C -\Gamma_p)\neq 0$. \\
For each $p$, let $(A_2)_p$ be the locus of $q \in A_2$ such that $\O_C(D_p\vert_C - \Gamma_q)$ is effective.  Then 
\[\bigcup_{p \in A_2} (A_2)_p = A_2, \]
by Proposition \ref{lin_series_prop}(4).
Further, by the upper semicontinuity of $ \dim H^0$, the locus $(A_2)_p$ is closed for any particular $p$, and since $h^1(S,\O_S)=0$, we have that $\Pic(S)$ is discrete and countable. As a result, there must be some single $p$ such that $(A_2)_p = A_2$.  Let $D = D_p$, so $\O_C(D\vert_C - \Gamma_q)$ is effective for all $q \in A_2$. 

The map $A_2 = (A_2)_p\ra W_{C\cdot D-2e}C$ sending a point $p\in A_2$ to the effective divisor class $D\vert_C -\Gamma_p\in W_{C\cdot D-2e}C$ is an embedding. Therefore, $W_{C\cdot D-2e}C$ contains an abelian subvariety, and so by minimality of $e$ we conclude that $C\cdot D-2e \geq e$, and hence 
\begin{equation}\label{bound_CD} C\cdot D \geq 3e. \end{equation}

Set $m_0=D^2/(C\cdot D)$.
As the curve $C$ is ample, the Hodge index theorem implies
\begin{equation}\label{CDCD}
C^2D^2\leq (C\cdot D)^2,
\end{equation}
and so $m_0 C^2 \leq C\cdot D$.
 Combining this with inequalities \ref{lin_series_prop}(2) and \ref{lin_series_prop}(3), respectively, we get
\begin{align}\label{m0_upper} m_0  &\leq \frac{C \cdot D}{C^2} < \frac{1}{2},\\
\label{m0_eqn} m_0 C^2(1 - m_0) &\leq C \cdot D(1-m_0) = D\cdot(C-D) \leq 2e.
 \end{align}
Furthermore, combining inequality \eqref{bound_CD} with \ref{lin_series_prop}(3) we have
\[3e(1-m_0) \leq C\cdot D(1-m_0) \leq 2e, \]
and so together with \eqref{m0_upper}, we have $1/3 \leq m_0 < 1/2$.

The function $m_0(1-m_0)$ is monotonically increasing in the range $[1/3,1/2)$, and so \eqref{m0_eqn} gives
\[\frac{2C^2}{9} \leq m_0(1-m_0)C^2 \leq 2e, \]
so we conclude $C^2\leq 9e$, which 
contradicts our hypothesis.
\end{proof}

For a very ample divisor $P$ on $S$, recall that we define the \defi{exceptional subset} with respect to $P$ to be  
\[\Exc_P \colonequals \big\{\text{integral classes $H$ in $\Amp(S)$ such that } H^2 \leq 9(H \cdot P-1)\big\}.\]

\begin{cor}\label{exceptional_cor}
Let $S/k$ be a smooth projective surface with $h^1(S, \O_S)=0$.
For any very ample divisor $P$, if $C \subset S$ is a smooth curve with class 
\[[C] \in \Amp(S) \smallsetminus \Exc_P,\] 
then for all $e < \gon_k(C)$, the locus $W_eC$ does not contain any positive-dimensional abelian varieties.
\end{cor}
\begin{proof}  Suppose that $e < \gon_k(C)$.  Then $\gon_k(C) \leq P \cdot C$, as exhibited by projection from a codimension $2$ plane in $\pp H^0(C, P|_C)^\vee$.  Therefore $e \leq P\cdot C -1$.  As $[C] \not\in \Exc_P$, we have that $P\cdot C -1 < C^2/9$.  Combining these inequalities gives $e < C^2/9$.
%\begin{align*}
%\gon_k(C) & \leq P \cdot C & \text{Lemma \ref{upper_bound_vample}}\\
%& \leq C^2/9 & C \not\in \Exc_P(\Amp(S)),
%\end{align*}
%As $C \not\in \Exc_P(\Amp(S))$, we necessarily have that $P \cdot C \leq C^2/9$,
Therefore, by Theorem \ref{main_geometric}, $(W_eC)_\cc$ (and hence $W_eC$) does not contain positive-dimensional abelian varieties.
% for $e < \gon_k(C)$.
%The result now follows, as $\gon_k(C) \leq P \cdot C$ by Lemma \ref{upper_bound_vample}.
\end{proof}

To imply the result stated in the introduction, we need the following elementary results about the intersection
\[\Exc_P(N) \colonequals \Exc_P\cap N,\]
for $N$ a closed subcone of the ample cone.

\begin{lem}\label{finiteness_1}
For any closed subcone $N$ and any very ample divisor $P$ on $S$, the set $\Exc_P(N)$ of exceptional classes in $N$ with respect to $P$ is finite.
\end{lem}

We will deduce this from the following elementary result.

\begin{lem}\label{cone_lemma}
Suppose that $N \subset \rr^n$ is a closed cone and let $f \colon N \to \rr$ be a continuous function taking positive values away from $0$.  Let $\Lambda$ be any lattice in $\rr^n$. If for all $H \in N$ and all $\lambda\geq 0$ , we have
\[f(\lambda H ) = \lambda f(H), \]
then for any $c\in \rr$, the set
\[
\{H\in N : f(H)\leq c\}\cap \Lambda
\]
is finite.
\end{lem}

\begin{proof}
Let $\ss$ be the unit sphere in $\rr^n$. Set $c_{min}=\inf\{f(H)\vert H\in \ss\cap N\}$. Since $\ss$ is compact and $N$ is closed, the intersection $\ss\cap N$ is compact.  So this minimum is achieved by $f$ on $\ss\cap N$, and in particular $c_{min}>0$. By the hypothesis $f(\lambda H ) = \lambda f(H)$, we then have that $f(H)>rc_{min}$ for all $H\in N\setminus B_r$, where $B_r$ is the closed ball of radius $r$. Then for any $c>0$, the set
\[
\{H\in N\vert f(H)\leq c\} 
\]
is a closed set contained in the compact set $B_{c/c_{min}}$, and is hence compact. So its intersection with the discrete set $\Lambda$ is finite.
\end{proof}

\begin{proof}[{Proof of Lemma \ref{finiteness_1}}]
If $H^2 \leq 9(P \cdot H -1)$, then $H^2 < 9H \cdot P$; hence it suffices to show that there are finitely many such integral classes $H$ in $N$.
Let $f \colon N  \smallsetminus \{0\} \to \rr$ be the continuous function
\[H \mapsto f(H) = \frac{H^2}{9 P \cdot H}. \]
As $H$ and $P$ are both ample, the function $f$ is positive and clearly satisfies $f(\lambda H) = \lambda f(H)$.  Therefore by Lemma \ref{cone_lemma} there are only finitely many integral classes $H$ for which
\[ f(H ) = \frac{H^2}{9 P \cdot H} < 1. \qedhere \]
\end{proof}

\subsection{Example: curves on \texorpdfstring{$\pp^1\times \pp^1$}{P1 X P1}}

When the divisor structure on $S$ is sufficiently well understood, our techniques allow one to explicitly compute the exceptional set for the entire ample cone.
%While the exceptional set in the statement of Theorem \ref{main_geometric} is not explicit, if $S$ has well-understood divisor structure, the proof of Theorem \ref{main_geometric} can be applied 
We present one example here.
%, working over $\cc$ to simplify the presentation.
\begin{prop}\label{p1_times_p1}
Let $S=\pp^1\times \pp^1$, and let $C$ be a smooth curve of any bidegree $(d_1,d_2)$ with $d_1\leq d_2$ and $(d_1,d_2)\neq (3,3)$ or $(2,2)$. Then, for $e<d_1$,  $W_eC$ contains no positive-dimensional abelian varieties.
\end{prop}
\begin{rem}\label{p1xp1_sharp}
The assumption that $(d_1,d_2)\neq (3,3)$ or $(2,2)$ in the proposition is necessary, as we now explain.  The smooth $(3,3)$ curves on $\pp^1 \times \pp^1$ (the complete intersection of a quadric and a cubic surface under the embedding of $\pp^1 \times \pp^1$ in $\pp^3$ by $\O(1,1)$) are canonical curves of genus $4$, and there exist bielliptic genus $4$ curves.  Explicitly, if the cubic surface is the cone over a smooth plane cubic and the quadric is general, then projection from the cone point gives a degree $2$ map from the curve to the cubic plane curve.
%a $(3,3)$ curve is bielliptic; a (3,3) curve can be expressed as the intersection of a quadric and a singular cubic surface in $\pp^3$, and projection from the singular point gives a degree 2 morphism from the curve to a plane cubic.

Likewise, if $(d_1,d_2)=(2,2)$, then $C$ is elliptic and $W_1C = C$.
\end{rem}
\begin{proof}
The cases of curves with $d_1\leq 1$ are trivial. If $d_1=2$, by assumption $d_2\geq 3$, so the genus of the curve is $d_2-1 >1$, and $W_1C$ contains no positive-dimensional abelian varieties. So we may assume $d_1\geq 3$ and $d_2\geq 4$. Let $C$ be a smooth curve of bidegree  $(d_1,d_2)$ with $d_1\geq 3$ and $d_2\geq 4$, and suppose the conclusion of the proposition fails for $C$. Let $e$ be minimal such that $W_eC$ contains a positive dimensional abelian variety $A$ so $e < d_1$. 
%Then the natural map $W_eC\times W_eC\ra W_{2e}C$ induces a map $A\times A\ra W_{2e}C$ with image $A_2\cong A$ consisting of divisors $\Gamma_a$ moving in basepoint-free pencils. 
By Lemma \ref{produce_bpf_pencil}, the points $p$ of $A_2$ give rise to basepoint free pencils $\Gamma_p$ of degree $2e$. Then
\[\deg \Gamma_p = 2e <2d_1\leq d_1d_2/2 = C^2/4.\]
%These divisors have degree $2e<2d_1$, and i
%In particular they satisfy the key hypothesis
%\[ \deg \Gamma_p < C^2/4\]
%since $C^2=2d_1d_2\geq 8 d_1$. 

So we apply Proposition \ref{lin_series_prop} to guarantee the existence of an  effective divisor $D$, say of class $(x,y)$ with $x,y \geq 0$, satisfying \\
\indent\ref{lin_series_prop}(2): $d_1y + d_2x = C \cdot D  < C^2/2 = d_1d_2$; \\
\indent\ref{lin_series_prop}(3): $d_1y + d_2x - 2xy = D\cdot(C-D) \leq 2e < 2d_1 \leq 2d_2$.\\
In exactly the same way as in the proof of Theorem \ref{every_value_geometric}, this forces $x,y \leq 1$.
%We may write these two inequalities as
%\[d_2(d_1-2x) + d_1(d_2-2y) > 0, \qquad (d_1-2x)(d_2-2y) > d_1(d_2-4) \geq 0, \]
%with the right-most inequality coming from our assumption that $d_2 \geq 4$.  Therefore both $(d_1-2x)$ and $(d_2-2y)$ must be nonnegative.  As $-2x$ is nonpositive and $(d_2-2y)$ is nonnegative, the second of these inequalities implies
%\[ d_1(d_2-2y) > d_1(d_2 -4) \qquad \Rightarrow y < 2.\]
%Similarly for $x$.
 %in the proposition together with the requirement $ C\cdot D\leq C^2/2$  
Thus $(x,y)$ is $(0,1)$, $(1,0)$, or $(1,1)$.  As in the proof of Theorem \ref{main_geometric}, there is a single choice of divisor class $D$ such that $\O_C(D|_C - \Gamma_p)$ is effective for all $p$.  In the first two cases, sending $\Gamma_p$ to the effective divisor of class $D|_C-\Gamma_p$, whose degree is $D\cdot C - 2e \leq D^2 = 2xy =0$ by \ref{lin_series_prop}(3), induces an isomorphism between $A_2$ and $W_{0}(C)=\mathrm{pt}$, which contradicts that $A$ is positive-dimensional. 

Now we consider the case $(x,y)=(1,1)$. By \ref{lin_series_prop}(3) we have the inequality
\[d_1+d_2-2 \leq 2e. \]
Combining this with 
\[2e \leq 2d_1 -2 \leq d_1 + d_2 -2\]
shows that equality must hold everywhere.  Therefore $d_1 = d_2$ and $e = d_1 -1$.
Now $D\cdot C - \deg \Gamma_p\leq D^2 = 2$, so we have an inclusion $A_2\ra W_2C$, so $W_2C$ contains a positive-dimensional abelian variety. This is a contradiction since we have $e= d_2-1\geq 3$ and assumed $e$ was minimal such that $W_eC$ contains an abelian variety.
\end{proof}

\section{Number-Theoretic Consequences}

Lang's general conjecture \cite[{$\S 3$, Statement 3.6}]{Lang} on rational points is known in its entirety for subvarieties of abelian varieties by the work of Faltings.

\begin{thm}[Faltings \cite{faltings1}]\label{faltings}
Let $k$ be a number field.  Let $X \subset A$ be a subvariety of an abelian variety $A$ over $k$.  Then there exist finitely many translates of abelian subvarieties 
\[Z_i = z_i + B_i,\qquad (B_i \subseteq A \text{ abelian subvariety})\]
that contain all of the rational points of $X$.  In particular, if $X(k)$ is infinite, then $X$ contains a translate of a positive-dimensional abelian subvariety of $A$.
\end{thm}

Recall that $C^{(e)}$ is the $e$th symmetric power of $C$, and $W_eC$ is the image of $C^{(e)} \to \Pic^eC$.
We will apply Faltings' theorem to $W_eC$.  

\begin{lem}\label{airr_eq}
Let $e_0 \leq \gon_k(C)$ be some positive integer.  If for all $e < e_0$, the subvariety $W_eC_{\bar{k}} \subseteq \Pic^e C_{\bar{k}}$ does not contain any positive-dimensional abelian varieties, then 
\[\airr_k(C) \geq e_0.\] 
In particular, if this holds with $e_0 = \gon_k(C)$, then $\airr_k(C) = \gon_k(C)$.
\end{lem}
\begin{proof}
If $e < \gon_k(C)$, then $C^{(e)}(k) \to W_eC(k)$ is injective, so $C^{(e)}(k)$ is finite if and only if $W_eC(k)$ is finite.
As $W_eC$ is a subvariety of the torsor $\Pic^eC$ of the abelian variety $\Pic^0C$, Faltings' theorem implies that the set of points $W_eC(L)$ is finite for all finite extensions $L/k$ if and only if $W_eC_{\bar{k}}$ does not contain any positive-dimensional abelian varieties. 
\end{proof}

\begin{rem}\label{geometric_property}
Lemma \ref{airr_eq} shows that $\airr_{\bar{k}}(C) \leq e$ if and only if $\gon(C_{\bar{k}}) \leq e$ or $W_fC_{\bar{k}}$ contains a positive-dimensional abelian subvariety for some $f\leq e$.  Thus $\airr_{\bar{k}}(C)$ depends only on $C_{\bar{k}}$.
\end{rem}

\begin{proof}[Proof of Theorem \ref{every_value}]
Suppose that $\gamma, \alpha$ are such that $0 < \gamma/2 \leq \alpha \leq \gamma$.  It suffices to find a nice curve $C$ over $\qq$ such that $\airr_{\qq}(C) = \airr_{\bar{\qq}}(C) = \alpha$ and $\gon_{\qq}(C) = \gon_{\bar{\qq}}(C) = \gamma$.  

For $\gamma = 1$, then $\alpha =1$ and we take $C = \pp^1_\qq$.  If $\gamma = 2$ and $\alpha = 1$, we may take $C$ to be an elliptic curve over $\qq$ of positive rank.  For $\gamma = 2$ and $\alpha = 2$, we may take $C$ to be any hyperelliptic curve of genus at least $2$.  For $\gamma =3$ and $\alpha = 2$, we may take any non-hyperelliptic curve that is a double cover of a positive-rank elliptic curve (see Remark \ref{p1xp1_sharp} for a construction in genus $4$).  For $\gamma = 3$ and $\alpha = 3$, we may take any non-hyperelliptic, non-bielliptic trigonal curve by the work of Harris-Silverman \cite[Corollary 3]{harris_silverman} and Lemma \ref{airr_eq} (e.g., a canonical curve of genus $4$ that is non-bielliptic).

Therefore, we may assume that $\gamma \geq 4$ (and so $\alpha\geq 2$).  Let $E$ be a positive rank elliptic curve over $\qq$.  By Theorem \ref{every_value_geometric} a smooth curve $C$ on $E\times \pp^1_{\qq}$ in numerical class $(\gamma, \alpha)$ has $\gon_{\bar{\qq}}(C) = \gamma$ and $\airr_{\bar{\qq}}(C) \geq \alpha$.  As $C$ has a map $\pi_1$ of degree $\alpha$ to $E$, we further have $\pi_1^{-1}(E(\qq)) \subseteq C_\alpha$, and so $\airr_{\bar{\qq}}(C) \leq \alpha$; therefore equality holds.
\end{proof}

\begin{proof}[Proof of Theorem\ref{main_q0}]
%Let $P_C$ be some very ample divisor on a curve $C$ over $k$.  Then the complete linear series $|P_C|$ gives rise to an embedding
%\[C\hookrightarrow \pp H^0(C, P_C)^\vee \colonequals \pp_k^r.\]
%Projecting from a $k$-rational codimension $2$ plane of $\pp^r_k$ exhibits $C$ as a $(\deg P_C)$-cover of $\pp_k^1$; therefore 
%\[\gon_k(C) \leq \deg P_C, \]
%for any very ample divisor $P_C$.

Suppose that $C \hookrightarrow S/k$ is a smooth ample curve.  
Let 
\[e_0 \colonequals \min\left( \gon_k(C), \frac{C^2}{9}\right). \]
Then by Theorem \ref{main_geometric}, $W_eC$ contains no positive-dimensional abelian varieties for $e < e_0 \leq \gon_k(C)$.  Therefore, by Lemma \ref{airr_eq}, we have that $\airr_k(C) \geq e_0$. 
Now let $P$ be any choice of very ample divisor on $S$.  If $9C \cdot P \leq C^2$ (i.e. $[C] \not\in \Exc_P$), then by Corollary \ref{exceptional_cor}, $W_eC$ contains no positive-dimensional abelian varieties for $e < \gon_k(C)$.  Therefore $\airr_k(C) = \gon_k(C)$ by Lemma \ref{airr_eq}.

On the other hand, for any closed subcone $N \subseteq \Amp(S)$, Lemma \ref{finiteness_1} guarantees that $\Exc_P \cap N$ is finite.
\end{proof}

\begin{proof}[Proof of Corollary \ref{main_pic1}]
If $\Pic(S_{\bar{k}}) \simeq \zz \cdot \O_S(1)$ for a very ample line bundle $\O_S(1)$ and $\O_S(C) \simeq \O_S(\alpha)$, then $[C] \not\in \Exc_{\O_S(1)}$ is equivalent to $\alpha \geq 9$.
\end{proof}

\begin{proof}[Proof of Corollary \ref{main_ci}]
The gonality of any complete intersection curve $C \subset \pp^r_k$ of type $(d_1, d_2, \dots, d_{r-1})$ is at most $\deg C = d_1 d_2 \cdots d_{r-1}$.  It therefore suffices by Lemma \ref{airr_eq} to show the nonexistence of abelian subvarieties in $W_e(C_\cc)$ for $e < d_1d_2 \cdots d_{r-1}$.
By \cite[Theorem 3.1]{ullery}, if $4 \leq d_1 < d_2 \leq \cdots \leq d_{n-1}$, then $C_\cc$ lies on a smooth complete intersection surface $S/\cc$ of type $(d_2, \cdots, d_{n-1})$ with $\Pic S \simeq \zz\cdot[\O_S(1)]$ and $[C]$ equals $\O_S(d_1)$.  Therefore, the result follows from Theorem \ref{main_pic1} with $P = \O_S(1)$, as $P \cdot C = d_1d_2 \cdots d_{r-1}$. 
\end{proof}

\begin{proof}[Proof of Theorem \ref{main_p1xp1}]
Let $C$ be a nice curve of type $(d_1, d_2)$ on $\pp^1_k \times \pp^1_k$ with $1 \leq d_1 \leq d_2$.  Then we claim that $\gon_k(C) = \gon_{\bar{k}}(C) = d_1$.  The upper bound is provided by the projection
\[ C \hookrightarrow \pp^1_k \times \pp^1_k \to \pp^1_k\]
onto the first factor.  As the tensor product of a $p$-very ample and a $q$-very ample bundle is $(p+q)$-very ample, the line bundle 
\[K_C = \O_C(d_1-2, d_2-2) = \O_C(d_1 - 2, d_1 - 2) \otimes \O_C(0, d_2 - d_1)\] 
is $(d_1-2)$-very ample (as $\O_C(0, d_2 - d_1)$ is either trivial or base point free.)  Therefore we have the lower bound $\gon_{\bar{k}}(C) \geq d_1$ (in fact even a weaker statement is true, see \cite[Lemma 1.3]{bdelu}).
%that $d_1 -1 $ points on $C$ always impose independent conditions on sections of 
By Proposition \ref{p1_times_p1}, $W_eC$ contains no positive-dimensional abelian subvarieties for $e <d_1$ as long as $(d_1, d_2) \neq (2,2)$ or $(3,3)$.  Therefore $C_{d_1-1}$ is finite for all such $(d_1, d_2)$ and $\airr_k(C) = \gon_k(C) = d_1$.
If $(d_1, d_2) = (2,2)$, then $C_{\bar{k}}$ is an elliptic curve and so $\airr_{\bar{k}}(C) = 1$.  
%If $(d_1, d_2) = (3,3)$, then $C$ is a canonical curve of genus $4$.  
%By Theorem \ref{quad_cubic_thm}, there exists a finite extension $K/k$ such that $(C_K)_2$ is infinite if and only if $C$ is bielliptic (as it is never hyperelliptic); in this case $\airr_{\bar{k}}(C) = 2$.  Otherwise $\airr_{\bar{k}}(C) = 3$.
If $(d_1,d_2)=(3,3)$, then $C$ is a canonical curve of genus 4, and in particular is not hyperelliptic.  If $C$ is bielliptic, then $(C_K)_2$ is infinite for any finite extension $K$ of $k$ over which the underlying genus 1 curve acquires infinitely many $K$-points, so $\airr_{\bar{k}}(C)=2$.  If $C$ is not bielliptic, then the work of Harris-Silverman \cite[Corollary 3]{harris_silverman} implies that $(C_K)_2$ is finite for every finite extension $K$ of $k$, so 
\[\airr_{\bar{k}}(C) \geq 3 = \gon_{\bar{k}}(C) \geq \airr_{\bar{k}}(C),\] 
and $\airr_{\bar{k}}(C)=3$.
\end{proof}
\subsection*{Funding}
This material is based upon work supported by the National Science Foundation Graduate Research Fellowship under Grant No. [DMS-1122374 to I.V.] and the National Science Foundation [DMS-1601946 to I.V.].
\subsection*{Acknowledgements} The second author's work on this subject was first inspired by hearing Lazarsfeld's proof of the minimal gonality of complete intersection curves in Brooke Ullery's class on linear series and hearing from Rachel Pries about the beautiful paper of Debarre--Klassen; we would like to thank them for these introductions.  We would also like to thank Joe Harris, Gerriet Martens, Bjorn Poonen, Brooke Ullery, Ravi Vakil, and Bianca Viray for helpful conversations, and to the anonymous referee for numerous corrections and suggestions.

\end{document}